\documentclass{amsart}

\usepackage{amsmath,amssymb}
\usepackage{graphicx}
\usepackage{epstopdf}
\usepackage{caption}
\usepackage{subcaption}
\usepackage{xcolor}
\usepackage{color}
\usepackage{sidecap}
\usepackage{pagecolor}
\usepackage{tikz}
\usepackage{colortbl} 
\usepackage{pdfpages}
\usepackage{amsthm}
\usepackage{hyperref}

\usepackage{marginnote}

\numberwithin{equation}{section}

\newtheorem{theorem}{Theorem}[section]

\newtheorem{proposition}[theorem]{Proposition}

\newtheorem{corollary}[theorem]{Corollary}
\newtheorem{definition}[theorem]{Definition}
\newtheorem{remark}[theorem]{Remark}

\newcommand{\R}{\mathbb{R}}
\newcommand{\C}{\mathbb{C}}

\newcommand{\N}{\mathbb{N}}

\newcommand{\D}{\mathcal{D}}

\newcommand{\mc}[1]{\mathcal{#1}}
\renewcommand{\epsilon}{\varepsilon}
\newcommand{\bn}{\mathbf{n}}

\newcommand*{\mailto}[1]{\href{mailto:#1}{\nolinkurl{#1}}}
\newcommand{\non}{\nonumber \\}

\renewcommand{\d}{\mathrm{d}}

\begin{document}
\title{A two-component nonlinear variational wave system}

\allowdisplaybreaks

\author[P. Aursand]{Peder Aursand}
\address{Aker BP ASA}

\author[A. Nordli]{Anders Nordli}
\address{Engineering Science and Safety\\ UiT\\ UiT The Arctic University of Norway\\ Troms\o \\ Norway}
\email{\mailto{anders.s.nordli@uit.no}}

\date{\today}

\begin{abstract}
We derive a novel two-component generalization of the nonlinear variational wave equation as a model for the director field of a nematic liquid crystal with a variable order parameter. The two-component nonlinear variational wave equation admits solutions locally in time. We show that a particular long time asymptotic expansion around a constant state in a moving frame satisfy the two-component Hunter--Saxton system.
\end{abstract}

\maketitle

\section{Introduction}
The nonlinear variational wave equation $\psi_{tt}-c(\psi)(c(\psi)\psi_x)_x=0$ was derived by Saxton \cite{S} as a model of the director field of a nematic liquid crystal. The nonlinear variational wave equation has recieved wide attention \cite{BCZ,BZ,GHZ,HR} due to mathematical challenges in the form of wavebreaking in finite time, and distinct ways to extend the solution to a global weak solution. In this context wavebreaking means that either the time or the space derivative becomes unbounded at certain points, while the solution remains H\" older continuous.

Furthermore, Hunter and Saxton derived the equation $u_{tx}+(uu_x)_x=\frac12u_x^2$ from the nonlinear variational wave equation \cite{HS} as an asymptotic equation for small perturbations in the long time regime in a moving frame. The Hunter--Saxton equation share many of the features of the nonlinear variational wave equation such as wavebreaking, conservative and dissipative weak solutions, and H\" older continuity \cite{BC,BHR,BZZ,CGH,D,HZ2,HZ3,ZZ,ZZ2}. In addition it exhibited novel features such as complete integrability and interpretation as a geodesic flow \cite{HZ,KM}. It also proved  easier to work with due to there being only one family of characteristics and existence of explicit solutions.

A two-component generalization for the Hunter--Saxton equation was derived from two-component Camassa--Holm \cite{DP}, and independently form the Gurevich--Zybin system \cite{P}. The two-component generalization is similar to the Hunter--Saxton equation since wavebreaking is also possible, there are both conservative and dissipative weak solutions, and $u$ is H\" older continuous \cite{GN,N,W}. However when the second variable is nonzero almost everywhere initially, there will be no wave breaking \cite{N}, and in that sense the introduction of a second variable regularizes the equation. The two-component Hunter--Saxton system has, however, not been shown to be related to the theory of nematic liquid crystals, and one of the aims of this paper is to establish that the two-component Hunter--Saxton system can indeed be derived from the theory of nematic liquid crystals, and that the new variable is related to the order parameter.

In the Ericksen--Leslie theory of nematic liquid crystals the configuration is described by a director field $\bn$ which gives the local orientation of the rods, and an order parameter field $s$ which gives the local degree of orientation \cite{E,L}. When the nonlinear variational wave equation was first derived the local degree of orientation was assumed constant, so that  a scalar equation was obtained \cite{S}. Here we will follow a similar route, but account for variable degree of orientation. Thus, in Section \ref{section:derivation}, we will arrive at a novel two-component system of nonlinear wave equations. Moreover, we show that a similar asymptotic expansion to the one by Hunter and Saxton \cite{HS} yields the two-component Hunter--Saxton equation. In Section \ref{section:analysis} existence of global solutions is shown in the case of constant wave speed, and local solutions is shown to exist in general. The proofs rely on fixed point iterations and standard semigroup theory for evolution equations \cite{K}. 

\section{Derivation of the equations} \label{section:derivation}
In Ericksen--Leslie theory a nematic liquid crystal in a domain $\Omega$ is desribed by a director field $\bn: [0,T]\times\Omega\rightarrow \R\mathbb{P}^2$ and an order parameter $s:[0,T]\times\Omega\rightarrow (-\frac12,1)$ \cite{E,L}. The Lagrangian of the system is then \cite{E} given by
\begin{equation}
\mc L = -(s\bn)_t^2+ W_2(s,\nabla s,\bn,\nabla\bn) + W_0(s),
\end{equation}
with the potential energy term $W_2$ 
\begin{align}
W_2(s,\nabla s,\bn,\nabla\bn) &= \left(K_1+L_1s\right)s^2\left(\nabla\cdot\bn\right)^2 + \left(K_2+L_2s\right)s^2\left(\bn\cdot\nabla\times\bn\right)^2\non
	&\quad+ \left(K_3+L_3s\right)s^2\left|\bn\times\nabla\times\bn\right|^2 \non 
	&\quad+ \left((K_2+K_4)+(L_2+L_4)s\right)s^2\left[\mathrm{tr}\:\nabla\bn^2-(\nabla\cdot\bn)^2\right]\non
	&\quad+ \left(\kappa_1+\lambda_1s\right)\left|\nabla s\right|^2 + \left(\kappa_2+\lambda_2s\right)\left(\nabla s\cdot\bn\right)^2 \non
	&\quad+ \left(\kappa_3+\lambda_3s\right)\left(\nabla\cdot\bn\right)\left(\nabla s\cdot\bn\right) + \left(\kappa_4+\lambda_4s\right)\nabla s\cdot\left((\bn\cdot\nabla)\bn\right),
\end{align}
and some function $W_0$ satisfying 
\begin{equation}
\label{eq:W_0 cond}
	\lim_{s\downarrow-\frac12}W_0(s)=+\infty,\quad \lim_{s\uparrow 1}W_0(s)=+\infty, \text{ and } \lim_{s\rightarrow 0}\frac{W_0(s)}{s^2} \in\R.
\end{equation}
We are interested in the case where $\bn=(\cos\psi,\sin\psi,0)$, and $s$ and $\psi$ depends on $t$ and $x$ only. 
Then the expression for $W_2$ can be written
\begin{align}
W_2 &= s^2\psi_x^2\left((K_1+L_1s)\sin^2\psi+(K_3+L_3s)\cos^2\psi\right)\non
	&\quad + s_x^2\left((\kappa_1+\lambda_1s)\sin^2\psi+(\tilde\kappa_2+\tilde\lambda_2s)\cos^2\psi\right)\non
	&\quad -s_x\psi_x\sin 2\psi\left(\frac{\kappa_3+\kappa_4}2+\frac{\lambda_3+\lambda_4}2s\right),
\end{align}
where $\tilde\kappa_2=\kappa_1+\kappa_2$ and $\tilde\lambda_2 = \lambda_1+\lambda_2$. We simplify the problem by considering the case where
\begin{subequations}
\begin{align}
\kappa_3 &= \kappa_4 = \lambda_3 = \lambda_4 = 0,\\
L_1 &= L_3 = \lambda_1 = \tilde\lambda_2 = 0,\\
K_1&= \kappa_1,\\
K_3 &= \tilde\kappa_2.
\end{align}
\end{subequations}
Then with the definition $c(\psi)^2= K_1\sin^2\psi + K_3\cos^2\psi$, we get the Lagrangian density
\begin{align}
\label{eq:2NVW lagr}
\mc L^{2NVW} &= -\frac 12s_t ^2 -\frac 12s^2\psi_t^2+\frac12s^2c(\psi)^2\psi_x^2+\frac 12c(\psi)^2s_x^2 + W_0(s)\non
	&= -\frac 12(s\bn)_t^2+\frac12 c(\bn)^2(s\bn)_x^2 + W_0(s).
\end{align}
\begin{definition}
We define the two-component nonlinear variational wave system to be
\begin{subequations}
\label{eq:2NVW}
\begin{align}
	s^2\left(\psi_{tt}-c(\psi)\left(c(\psi)\psi_x\right)_x\right)+2s\left(\psi_ts_t-c(\psi)^2\psi_xs_x\right)+c(\psi)c'(\psi)  s_x^2 &= 0,\\
	s_{tt} - c(\psi)\left(c(\psi) s_x\right)_x -c(\psi)c'(\psi)\psi_xs_x -s\left(\psi_t^2-c(\psi)^2\psi_x^2\right) + W_0'(s) &=0,\\
	c(\psi)^2 - K_1\sin^2\psi-K_3\cos^2\psi &= 0,\\
	W_0 &\in C^4\left(\left(-\frac12,1\right)\right).
\end{align}
\end{subequations}
\end{definition}
\begin{remark}
The system \eqref{eq:2NVW} is the Euler--Lagrange equations for \eqref{eq:2NVW lagr}.
\end{remark}
One can define a conserved energy for classical solutions.
\begin{proposition}
Define the energy density 
\begin{equation}
	\label{def:E}
	\mc E = \frac12\left(s^2(\psi_t^2+c(\psi)^2\psi_x^2)+(s_t^2+c(\psi)^2s_x^2)\right)+W_0(s).
\end{equation}
and the energy density flux
\begin{equation}
	\label{def:F}
	\mc F = \left(s^2\psi_t\psi_x+s_ts_x\right).
\end{equation}
Then for classical solutions of \eqref{eq:2NVW} the energy satisfy the equations
\begin{subequations}
\label{eq:cons E}
\begin{align}
	\mc E_t-\left(c(\psi)^2\mc F\right)_x &= 0,\\
	\mc F_t - \left(\mc E-2W_0(s)\right)_x &= 0.
\end{align}
\end{subequations}
\end{proposition}
We will now derive the two-component Hunter--Saxton system from \eqref{eq:2NVW lagr}. To follow the work of Hunter and Saxton \cite{HS} we introduce $\psi(t,x) = \psi_0 + \epsilon u(\epsilon t,x-c_0t)$ and $s = s_0+\epsilon r(\epsilon t,x-c_0t)$ with $c_0^2 = K_1\sin^2\psi_0 + K_3\cos^2\psi_0$. Then expansion of \eqref{eq:2NVW lagr} in powers of $\epsilon$ gives
\begin{align}
\mc L^\epsilon &= W_0(s_0) + \epsilon W_0'(s_0)r + \epsilon^2\frac12W_0''(s_0)r^2\non
&\quad + \epsilon^3\bigg(c_0r_tr_x+s_0^2c_0u_tu_x + s_0^2\left(cc'\right)_0 uu_x^2 + \left(cc'\right)_0 ur_x^2 + \frac16W_0'''(s_0)r^3\bigg)+ O(\epsilon^4).
\end{align}
If we select $s_0$ such that $W_0'(s_0) = W_0''(s_0) = W'''(s_0) = 0$, the third order terms in the above Lagrangian gives (possibly by rescaling)
\begin{equation}
\label{eq:2HS lagr}
\mc L^{2HS} = u_tu_x +uu_x^2 + r_tr_x + ur_x^2.
\end{equation}
\begin{theorem}
	The two-component Hunter--Saxton system
	\begin{subequations}
	\begin{align}
	(u_t+uu_x)_x &= \frac 12u_x^2+\frac 12\rho^2,\\
	\rho_t + (u\rho)_x &= 0,
	\end{align}
	\end{subequations}
	is the Euler--Lagrange equations for the Lagrangian density \eqref{eq:2HS lagr} with $\rho=r_x$.
\end{theorem}

\section{The two-component nonlinear variational wave system} \label{section:analysis}
We will first consider the case $K_1=K_3$, that is the wave speed $c$ is independent of $\psi$. We will assume $s \geq 0$ such that we can introduce the complex variable $\zeta = se^{i\psi}$, and \eqref{eq:2NVW} reduces to
\begin{equation}
\label{eq:2NVWlinear}
\zeta_{tt}-c^2\zeta_{xx} +\frac{W_0'\left(|\zeta|\right)}{|\zeta|}\zeta = 0,
\end{equation}
which is a defocusing nonlinear wave equation. Recall that the physical interpretation of $\psi$ is an angle, and thus any solutions $\psi$ and $\psi + n\cdot 2\pi$ should be considered equal from an application point of view. We have conservation laws for energy $\mc E = \frac12|\zeta_t|^2+\frac12c^2|\zeta_x|^2+W_0(|\zeta|)$ and energy flux $\mc F = \frac12\left(\bar\zeta_t\zeta_x+\zeta_t\bar\zeta_x\right)$ as follows
\begin{subequations}
\begin{align}
	\frac{\partial}{\partial t}\mc E - c^2\frac{\partial}{\partial x}\mc F &= 0,\\
	\frac{\partial}{\partial t}\mc F - \frac{\partial}{\partial x}\left(\mc E-2W_0(|\zeta|)\right) &= 0.
\end{align}
\end{subequations}
Hence both $\int_\R\mc E\:\d x$ and $\int_\R\mc F\:\d x$ are conserved for any classical solution with bounded energy. To analyze solutions we will need some conditions on the function $W_0$. We will assume that $W_0(s)\rightarrow\infty$ rapidly enough as $s\rightarrow1$ to be able to bound $\|W_0(|\zeta|)\|_\infty$ in terms of the total energy $E$. In addition we will assume that $W_0$ is non-negative and that $W_0$ is well behaved close to $s=0$, and also close to zeros of $W_0$. We will also relax \eqref{eq:W_0 cond} slightly to allow for $s=0$ to be a local maximum or minimum instead of a global minimum, in accordance with \cite{E}.
\begin{definition}
\label{def:W_0}
We will assume that $W_0\in C^4([0,1))$ and non-negative, and that there exists a finite number of $s^*$ such that $W_0(s^*)=0$, and for all zeros $s^*$,
\begin{align}
	\lim_{s\rightarrow 0}\frac{W_0(s)-W_0(0)}{s^2} &= \frac 12W_0''(0)\in\R,\\
	\lim_{s\rightarrow s^*}\frac{W_0(s)}{(s-s^*)^2} &= \frac12W_0''(s^*)\in [0,\infty).
\end{align}
Moreover, we let $W_0$ satisfy
\begin{equation}
	\int_s^1W_0(u)(1-u)\:d u = \infty,
\end{equation}
and that there exists $\tilde s\in[0,1)$ such that $W_0(s)>0$ and $W_0'(s)>0$ for all $s\in (\tilde s,1)$. 
\end{definition}
We define the function spaces for the solutions in the next definition.
\begin{definition}
\label{def:spaces}
	Let $\zeta^*\in\C$, $|\zeta^*|<1$, and
	\begin{equation}	
		X_{\zeta^*} = \{(\zeta,\sigma)\mid \zeta \in W^{1,\infty}(\R),\zeta-\zeta^*\in H^1(\R), \sigma \in L^\infty(\R)\cap L^2(\R)\},
	\end{equation}	
	and define
	\begin{align}
		\mc E(\zeta,\sigma) &= \frac12|\sigma|^2+\frac12|\zeta_x|^2+W_0(|\zeta|),\\
		\mc F(\zeta,\sigma) &= \frac12\sigma\bar\zeta_x+\frac12\bar\sigma\zeta_x,
	\end{align}
	and note that both $\mc E$ and $\mc F$ are real valued. Let now
	\begin{equation}
		X_E = \{(\zeta,\sigma)\in X_{\zeta^*}\mid |\mc E\|_{L^1} \leq E\},
	\end{equation}
	and denote
	\begin{align}
		\D_{T,E} &= \left\{\zeta\in C([0,T],W^{1,\infty}(\R,\C))\cap C^1([0,T],L^\infty(\R,\C) \mid (\zeta(t),\zeta_t(t))\in X_E\right\}.
	\end{align}
	In the case $T<\infty$ we equip $\D_{T,E}$ with the metric $d_{\D_{E,T}}$ induced from 
	\begin{align}
		\|\zeta\| &= \sup_{t\in[0,T]}\big(\|\zeta(t)\|_{W^{1,\infty}(\R)} +\| \zeta(t)-\zeta^*\|_{H^1(\R)}\non &\quad+ \|\zeta_t(t))\|_{L^\infty(\R)} + \| \zeta_t(t))\|_{L^2(\R)} + \|W_0(|\zeta(t)|)\|_{L^1(\R)}\big).
	\end{align}
\end{definition}
For bounded energy we can now a priori put constraints on possible solutions.
\begin{proposition}
\label{prop:W_0}
Let $W_0$ satisfy the conditions of Definition \ref{def:W_0}. Then there exist positive constants $c_E, C_E$, depending on $W_0$ and $E$ only, such that for any $(\zeta,\sigma)\in X_E$ we have
\begin{align}
	\|W_0(|\zeta|)\|_{L^\infty(\R)} &\leq C_E,\non
	\|\zeta\|_{L^\infty(\R)} &\leq c_E < 1.\nonumber
\end{align}
Moreover for $s\in [0,c_E]$ there is a positive constant $k_E$ such that
\begin{equation}
	W_0'(s)^2 \leq k_E W_0(s).
\end{equation}
Furthermore, we can define
\begin{align}
	L_E &= \sup_{s\in[0,c_E]} |W_0'(s)|,\\
	L_E' &= \sup_{s\in[0,c_E]} |\frac{W_0'(s)}{s}|,\\
	L_E'' &= \sup_{s\in[0,c_E]} |W_0''(s)|.
\end{align}
\end{proposition}
\begin{proof}
Since $\frac12\|\sigma\|_2^2+\frac12 c^2 \|\zeta_x\|_2^2 + \|W_0(|\zeta|)\|_1 \leq E$ we have that $|\zeta(x_1)-\zeta(x_2)|\leq\frac{\sqrt{2E_1}}{c}\sqrt{|x_1-x_2|}$ where $E_1 = E-\frac12\|\sigma\|_2^2-\|W_0(|\sigma|)\|_1 \geq 0$. Since $\zeta$ is continuous we have that there exists $\hat x$ such that $|\zeta(\hat x)| = \|\zeta\|_{L^\infty(\R)}\leq 1$. Then we have that
\begin{equation}
	|\zeta(x)| \geq 
		\begin{cases}
			\|\zeta\|_\infty-\frac{\sqrt{2E_1}}c\sqrt{|x-\hat x|},&\hat x-\frac{\|\zeta\|_\infty^2c^2}{2E_1}<x<\hat x+\frac{\|\zeta\|_\infty^2c^2}{2E_1},\\
			0, &\text{ else}.
		\end{cases}
\end{equation}
In particular, we have that
\begin{align}
	\int_\R W_0(|\zeta(x)|)\:\d x &\geq \int_{\hat x-\frac{c^2}{2E_1}(\|\zeta\|_\infty-\tilde s)^2}^{\hat x+\frac{c^2}{2E_1}(\|\zeta\|_\infty-\tilde s)^2} W_0(|\zeta(x)|)\:\d x\non
		&\geq \int_{\hat x-\frac{c^2}{2E_1}(\|\zeta\|_\infty-\tilde s)^2}^{\hat x+\frac{c^2}{2E_1}(\|\zeta\|_\infty-\tilde s)^2} W_0\left(\|\zeta\|_\infty-\frac{\sqrt{2E_1}}{c}\sqrt{|x-\hat x|})\right)\:\d x\non
		&= \int_{\tilde s}^{\|\zeta\|_\infty}W_0(u)\left(\|\zeta\|_\infty-u\right)\:\d u
\end{align}
By the monotone convergence theorem we have that
\begin{equation}
\lim_{S\rightarrow 1^-}\int_s^S W_0(u)\left(S-u\right)\:\d u = \int_s^1W_0(u)\left(1-u\right)\:\d u = \infty.
\end{equation}
Hence we have the inequality
\begin{equation}
	\int_{\tilde s}^{\|\zeta\|_\infty}W_0(u)\left(\|\zeta\|_\infty-u\right)\:\d u \leq E - E_1,
\end{equation}
which proves the existence of $c_E, C_E$ such that $\|\zeta\|_\infty = c_E < 1$ and $\|W_0(|\zeta|)\|_\infty = W_0(\|\zeta\|_\infty) = C_E<\infty$. The constant can be chosen to depend on $W_0$ and $E$ only.

To prove that $W_0'^2\leq k_E W_0$, note that whenever $W_0$ tends to zero that
\begin{equation}
	\lim_{s\rightarrow s^*}\frac{W_0'(s)^2}{W_0(s)} = 2W_0''(s^*)\in[0,\infty).
\end{equation} 
In between zeros of $W_0$ the fraction $\frac{W_0'^2}{W_0}$ is continuous. Since there is only a finite number of intervals where $W_0$ is zero and $[0,c_E]$ is bounded, $\frac{W_0'^2}{W_0}$ has to be bounded as well by a constant dependent on $c_E$. Since $W_0$ is non-negative we get the desired inequality. 
\end{proof}
\begin{remark}
\label{rmk:W_0}
Note that $\zeta_1$ and $\zeta_2$ in $X_E$ we have that
\begin{align}
	\int_\R W_0(|\zeta_1(x)|)\:\d x &= \int_\R\frac12L_E''|\zeta_1-\zeta^*|^2\:\d x \non
	&\leq \frac12L_E''\|\zeta_1-\zeta^*\|_2^2, 
\end{align}
and similarly, since $W_0(|\zeta|) = \int_{|\zeta^*|}^{|\zeta|}W_0''(u)(|\zeta|-u)\:\d u$,
\begin{align}
	\int_\R|W_0(|\zeta_1|)-W_0(|\zeta_2|)|\:\d x &=\int_\R\Bigg| \int_{|\zeta^*|}^{|\zeta_1|}W_0''(u)\left(|\zeta_1|-|\zeta_2|\right)\:\d u\non&\qquad+\int_{|\zeta_2|}^{|\zeta_1|}W_0''(u)\left(|\zeta_2|-u\right)\:\d u\Bigg|\:\d x \non
		&\leq L_E''\|\zeta_1-\zeta_2\|_2^2 + L_E''\|\zeta_1-\zeta^*\|_2\|\zeta_1-\zeta_2\|_2.
\end{align}
\end{remark}
\begin{theorem}
\label{thm:global}
Let $(\zeta_0,\zeta_{t0})\in X_E$, and assume that $W_0$ satisfies the conditions in Definition \ref{def:W_0}. Then there exists a unique global strong energy conserving solution $(\zeta,\zeta_t)\subseteq \D_{\infty,E}$ of \eqref{eq:2NVWlinear} depending continuously on the initial data $(\zeta_0,\zeta_{t0})\in X_E$.
\end{theorem}
\begin{proof}
The idea of the proof is to show that the solution in strong form
\begin{align}
	\label{eq:strong form}
	\zeta(t,x) &= \frac12\left(\zeta_0(x-ct)+\zeta_0(x+ct)\right) + \frac1{2c}\int_{x-ct}^{x+ct}\zeta_{t0}(y)\:\d y\non&\quad - \frac1{2c}\int_0^t\int_{x-c(t-s)}^{x+c(t-s)}\frac{W_0'(|\zeta(s,y)|)}{|	\zeta(s,y)|}\zeta(s,y)\:\d y\d s,
\end{align}
is the unique fixed point of a contraction on $\D_{T,E}$, as is done for Picard iteration for a to be specified $T$. To be able to prove that energy is conserved by strong solutions we will at first assume that the initial data in $X_E$ is smooth, $(\zeta_0,\zeta_{t0})\in C^2(\R)\times C^1(\R)$. We will then prove that the strong solution with smooth initial data preserve energy, and that the solution operator taking initial data in $X_E$ to solutions in $\D_{T,E}$ is continuous and can thus be extended to all of $X_E$. We prove that the energy is still conserved for these solutions. Finally we note that the existence time $T$ depends on $W_0$ and $E$ only, and thus that the existence time can be extended indefinitely.

We will now assume that $0<E<E'$ and that $(\zeta_0,\zeta_{t0})\in X_{E}$ and $\zeta_0\in C^2(\R), \zeta_{t0}\in C^1(\R)$, and let $(\hat\zeta,\hat\zeta_t)\in\D_{T,E'}\cap C^2([0,\infty)\times\R)$ with $(\hat\zeta(0),\hat\zeta_t(0))=(\zeta_0,\zeta_{t0})$. Then let $(\zeta,\zeta_t)\in C([0,T],X_{E'})$, $\zeta\in C^2([0,\infty)\times\R)$, be the classical solution of
\begin{equation}
	\label{eq:linear iterative}
	\zeta_{tt}-c^2\zeta_{xx}=-\frac{W'_0(|\hat\zeta|)}{|\hat\zeta|}\hat\zeta,\quad (\zeta,\zeta_t)|_{t=0}=(\zeta_0,\zeta_{t0}).
\end{equation}
Then by Duhamel's principle we have
\begin{subequations}
\label{eq:iterative solutions}
\begin{align}
	\zeta(t,x) &= \frac12\left(\zeta_0(x-ct)+\zeta_0(x+ct)\right) + \frac1{2c}\int_{x-ct}^{x+ct}\zeta_{t0}(y)\:\d y\non &\quad - \frac1{2c}\int_0^t\int_{x-c(t-s)}^{x+c(t-s)}\frac{W_0'(|\hat
		\zeta(s,y)|)}{|\hat\zeta(s,y)|}\hat\zeta(s,y)\:\d y\:\d s,\\
	\zeta_t(t,x) &=  \frac c2\left(\zeta_{x0}(x+ct)-\zeta_{x0}(x-ct)\right) + \frac12\left(\zeta_{t0}(x+ct)+\zeta_{t0}(x-ct)\right)\:\d y\non &\quad - \frac12\int_0^t\frac{W_0'(|\hat
		\zeta(s,x-c(t-s))|)}{|\hat\zeta(s,x-c(t-s))|}\hat\zeta(s,x-c(t-s))\non &\qquad\qquad+\frac{W_0'(|\hat\zeta(s,x+c(t-s))|)}{|\hat\zeta(s,x+c(t-s))|}\hat\zeta(s,x+c(t-s))\:\d s.
\end{align}
\end{subequations}
To prove that the strong solution \eqref{eq:iterative solutions} is a member of $\D_{T,E'}$ we will need norm estimates on the solutions. For $(\zeta,\zeta_t) \in \D_{T,E}$ we have that
\begin{align}
	\left|\frac1{2c}\int_{x-c(t-s)}^{x+c(t-s)} \frac{W_0'(|\zeta(s,y)|)}{|\zeta(s,y)|}\zeta(s,y)\:\d y\right| &\leq L_E (t-s),\\
	\left|\frac{\partial}{\partial x} \frac1{2c}\int_{x-c(t-s)}^{x+c(t-s)} \frac{W_0'(|\zeta(s,y)|)}{|\zeta(s,y)|}\zeta(s,y)\:\d y\right| &\leq \frac1c L_E.
\end{align}
The integrals can be interpreted as convolutions
\begin{equation}
	\int_{x-ct}^{x+ct} f(y)\:\d y = \mathbf{1}_{[-ct,ct]}*f (x),
\end{equation}
and then by Young's inequality we get for any $p\in[1,\infty]$, 
\begin{equation}
	\label{eq:conv}
	\left\|\int_{x-ct}^{x+ct} f(y)\:\d y\right\|_p \leq 2ct\|f\|_p.
\end{equation}
Thus, the solution $(\zeta,\zeta_t)$ satisfies for $p=2,\infty$ the estimates
\begin{align}
	\|\zeta(t)-\zeta^*\|_p &\leq \|\zeta_0-\zeta^*\|_p + t\|\zeta_{t0}\|_p + \frac12t^2\max\{ L_{E'},\sqrt{k_{E'}E'}\},\\
	\|\zeta_x(t)\|_p &\leq \|\zeta_{x0}\|_p + \frac1c\|\zeta_{t0}\|_p + \frac2ct\max\{ L_{E'},\sqrt{k_{E'}E'}\},\\
	\|\zeta_t(t)\|_p &\leq c\|\zeta_{x0}\|_p + \|\zeta_{t0}\|_p + \frac12t^2\max\{ L_{E'},\sqrt{k_{E'}E'}\}
\end{align}
We need to show that the energy is bounded. We have that
\begin{align}
\label{eq:iterative energy}
	&\frac{\partial}{\partial t}\left(\frac12|\zeta_t|^2+\frac12c^2|\zeta_x|^2+W_0(|\zeta|)\right) - \frac{\partial}{\partial x}\frac12c^2\left(\bar\zeta_t\zeta_x+\zeta_t\bar\zeta_x\right) =\non &\quad \frac12\zeta_t\left(\frac{W_0'(|\zeta|)}{|\zeta|}\bar\zeta - \frac{W_0'(|\hat\zeta|)}{|\hat\zeta|}\bar{\hat\zeta}\right)+\frac12\bar\zeta_t\left(\frac{W_0'(|\zeta|)}{|\zeta|}\zeta - \frac{W_0'(|\hat\zeta|)}{|\hat\zeta|}\hat\zeta\right).
\end{align}
Integrating in space and using the above estimates gives, whenever $E(t)<E'$,
\begin{align}
	\frac{\d}{\d t} E(t) &\leq \frac12\int_\R\left|\zeta_t\right|\left|\frac{W_0'(|\zeta|)}{|\zeta|}\bar\zeta - \frac{W_0'(|\hat\zeta|)}{|\hat\zeta|}\bar{\hat\zeta}\right|\:\d x\non &\quad+\frac12\int_\R\left|\bar\zeta_t\right|\left|\frac{W_0'(|\zeta|)}{|\zeta|}\zeta - \frac{W_0'(|\hat\zeta|)}{|\hat\zeta|}\hat\zeta\right|\:\d x\non
		&\leq \|\zeta_t(t)\|_2\left(\|W_0'(|\zeta(t)|)\|_2 +\|W_0'(|\hat\zeta(t)|)\|_2\right)\non
		&\leq \sqrt{E(t)}\left(\sqrt{k_{E'}\|W_0(|\zeta(t)|)\|_1}+\sqrt{k_{E'}\|W_0(|\hat\zeta(t)|)\|_1}\right)\non
		&\leq 2\sqrt{k_{E'}}E'
\end{align}
which yields
\begin{equation}
	E(t) \leq E+2\sqrt{k_{E'}}E't.
\end{equation}
Thus for $T\leq \frac1{2\sqrt{k_{E'}}}\left(1-\frac{E}{E'}\right)$ we have that the solution map $\Phi:\hat\zeta \mapsto \zeta$ of \eqref{eq:linear iterative} maps smooth functions in $\D_{T,E'}$ to smooth functions in $\D_{T,E'}$.

To prove that $\Phi$ is a contraction we will need further estimates on the integrals in \eqref{eq:iterative solutions}. For $\zeta_1,\zeta_2 \in \C$ with $0<|\zeta_1|<|\zeta_2|<1$ we have
\begin{align}
	\left|\frac{W'(|\zeta_1|)}{|\zeta_1|}\zeta_1 -  \frac{W'(|\zeta_2|)}{|\zeta_2|}\zeta_2\right| &\leq \left|\frac{W'(|\zeta_1|)}{|\zeta_1|}\left(\zeta_1 - \frac{|\zeta_1|}{|\zeta_2|}\zeta_2\right) + 		\frac{\zeta_2}{|\zeta_2|}\left(W'(|\zeta_1|) - W'(|\zeta_2|)\right)\right|\non
	&\leq \left|\frac{W'(|\zeta_1|)}{|\zeta_1|}\right||\zeta_1-\zeta_2| + \sup_{\zeta\in(\zeta_1,\zeta_2)}\left|W''(\zeta)\right|\left|\zeta_1-\zeta_2\right|
\end{align}
Then for $(\zeta_1,\zeta_{1t}),(\zeta_2,\zeta_{2t})\in\D_{T,E}$ we get
\begin{equation}
	\label{eq:rhs diff}
	\left|\frac{W'(|\zeta_1(t,x)|)}{|\zeta_1(t,x)|}\zeta_1(t,x) -  \frac{W'(|\zeta_2(t,x)|)}{|\zeta_2(t,x)|}\zeta_2(t,x)\right| \leq (L_E'+L_E'')|\zeta_1(t,x)-\zeta_2(t,x)|.
\end{equation}
Hence
\begin{align}
\label{eq:diff Linf}
	&\left|\frac1{2c}\int_{x-c(t-s)}^{x+c(t-s)}\frac{W_0'(|\zeta_1(s,y)|)}{|\zeta_1(s,y)|}\zeta_1(s,y) -  \frac{W_0'(|\zeta_2(s,y)|)}{|\zeta_2(s,y)|}\zeta_2(s,y)\:\d y\right|\non &\qquad\leq (t-s)(L_E'+L_E'')\|\zeta_1(s)-\zeta_2(s)\|_{L^\infty(\R)},\\
\label{eq:diff der Linf}
	&\left|\frac{\partial}{\partial x}\frac1{2c}\int_{x-c(t-s)}^{x+c(t-s)}\frac{W'(|\zeta_1(s,y)|)}{|\zeta_1(s,y)|}\zeta_1(s,y) -  \frac{W'(|\zeta_2(s,y)|)}{|\zeta_2(s,y)|}\zeta_2(s,y)\:\d y\right|\non &\qquad\leq 2(L_E'+L_E'')\|\zeta_1(s)-\zeta_2(s)\|_{L^\infty(\R)}.
\end{align}
We will need $L^2(\R)$-estimates in $x$ as well. For $(\zeta_1,\zeta_{1t}),(\zeta_2,\zeta_{2t})\in\D_{T,E}$ we get from \eqref{eq:rhs diff} and \eqref{eq:conv}
\begin{align}
\label{eq:L2 estimate}
	&\quad\left\|\int_{x-c(t-s)}^{x+c(t-s)}\frac{W_0'(|\zeta_1(s,y)|)}{|\zeta_1(s,y)|}\zeta_1(s,y) -  \frac{W_0'(|\zeta_2(s,y)|)}{|\zeta_2(s,y)|}\zeta_2(s,y)\:\d y\right\|_2\non
	&\leq (L_E'+L_E'')2c(t-s)\|\zeta_1(s,\cdot)-\zeta_2(s,\cdot)\|_2.
\end{align}
To prove that $\Phi$ is a contraction let $\hat\zeta_1,\hat\zeta_2 \in \D_{T,E'}\cap C^2([0,T]\times\R)$ with coinciding initial data $(\zeta_0,\zeta_{t0})$, and denote $\zeta_1=\Phi(\hat\zeta_1),\zeta_2 = \Phi(\hat\zeta_2)$. Then by \eqref{eq:diff Linf}, \eqref{eq:diff der Linf}, \eqref{eq:L2 estimate}, and \eqref{eq:iterative solutions}, we have for $p=2,\infty$,
\begin{align}
	\sup_{t\in[0,T]}\|\zeta_1(t) -\zeta_2(t)\|_{L^p(\R)} &\leq \left(L_{E'}'+L_{E'}''\right)\int_0^t(t-s)\|\hat\zeta_1(s)-\hat\zeta_2(s)\|_p\:\d s,\\
	\sup_{t\in[0,T]}\|\zeta_{1,t}(t) -\zeta_{2,t}(t)\|_{L^p(\R)} &\leq \left(L_{E'}'+L_{E'}''\right)\int_0^t\|\hat\zeta_1(s)-\hat\zeta_2(s)\|_p\:\d s,\\
	\sup_{t\in[0,T]}\|\zeta_{1,x}(t) -\zeta_{2,x}(t)\|_{L^p(\R)} &\leq \frac1c\left(L_{E'}'+L_{E'}''\right)\int_0^t\|\hat\zeta_1(s)-\hat\zeta_2(s)\|_p\:\d s.
\end{align}
We have that
\begin{align}
	\frac\d{\d t}\left(W_0(|\zeta_1|)-W_0(|\zeta_2|)\right) &= W_0'(|\zeta_1|)\frac{\bar\zeta_1\zeta_{1,t} + \zeta_1\bar\zeta_{1,t} }{2|\zeta_1|}\non &\quad -  W_0'(|\zeta_2|)\frac{\bar\zeta_2\zeta_{2,t} + \zeta_2\bar\zeta_{2,t}}{2|\zeta_2|}\non
	&= \left(W_0'(|\zeta_1|)- W_0'(|\zeta_2|)\right)\frac{\bar\zeta_1\zeta_{1,t} + \zeta_1\bar\zeta_{1,t}}{2|\zeta_1|}\non &\quad  + W_0'(|\zeta_2|)\left(\frac{\bar\zeta_1\zeta_{1,t} + \zeta_1\bar\zeta_{1,t}}{2|\zeta_1|} - \frac{\bar\zeta_2\zeta_{2,t} + \zeta_2\bar\zeta_{2,t}}{2|\zeta_2|}\right),
\end{align}
Further note that
\begin{align}
	 \left(\frac{\bar\zeta_1\zeta_{1,t} + \zeta_1\bar\zeta_{1,t}}{2|\zeta_1|} - \frac{\bar\zeta_2\zeta_{2,t} + \zeta_2\bar\zeta_{2,t}}{2|\zeta_2|}\right) &= \frac{1}{|\zeta_2|}\bigg[\frac12\left(|\zeta_2|-|\zeta_1|\right)(\frac{\bar\zeta_1}{|\zeta_1|}\zeta_{1,t}+\frac{\bar\zeta_1}{|\zeta_1|}\bar\zeta_{1,t}) \non &\quad + \frac12\big((\bar\zeta_1-\bar\zeta_2)\zeta_{1,t}+\bar\zeta_2(\zeta_{1,t}-\zeta_{2,t})\non &\quad+(\zeta_1-\zeta_2)\bar\zeta_{1,t}+\zeta_2(\bar\zeta_{1,t}-\bar\zeta_{2,t})\big)\bigg],
\end{align}
and hence
\begin{align}
\label{eq:diff W_0}
	\frac\d{\d t}\left\|W_0(|\zeta_1|)-W_0(|\zeta_2|)\right\|_1 &\leq \int_\R\left|W_0'(|\zeta_1|)- W_0'(|\zeta_2|)\right|\left|\frac{\bar\zeta_1\zeta_{1,t} + \zeta_1\bar\zeta_{1,t}}{2|\zeta_1|}\right|\:\d x\non &\quad+2\int_\R\left|\frac{W_0'(|\zeta_2|)}{|\zeta_2|}\right|\left|\zeta_1-\zeta_2\right|\left|\zeta_{1,t}\right|\:\d x \non &\quad+  \int_\R\left|\frac{W_0'(|\zeta_2|)}{|\zeta_2|}\right|\left|\zeta_{1,t}-\zeta_{2,t}\right||\zeta_2|\:\d x\non
	&\leq \left(L_{E'}''+2L_{E'}'\right)\sqrt{E'}\|\zeta_1 - \zeta_2\|_2 + \sqrt{k_{E'}E'}\|\zeta_{1,t}-\zeta_{2,t}\|_2.
\end{align}

Hence for $T$ small enough $\Phi$ is a contraction on the subspace of $\D_{T,E'}$ consisting of $\zeta\in C^2([0,T]\times\R)$. Thus there is a unique fixed point in $\D_{T,E'}$ which is a strong solution in the sense of \eqref{eq:strong form}. The limit may not be $C^2$, however. We want to prove that energy is actually conserved. From the construction of the metric on $\D_{T,E'}$ we see that the energy converges. A more careful investigation of \eqref{eq:iterative energy} reveals that
\begin{align}
\label{eq:energy cons}
	E(t) - E(0)&= \frac 12\int_0^t \int_\R \frac{W_0'(|\hat\zeta(s,y)|)}{|\hat\zeta(s,y)|}\bar{\hat\zeta}(s,y)(\hat\zeta_t(s,y)-\zeta_t(s,y)) \non &\qquad+\frac{W_0'(|\hat\zeta(s,y)|)}{|\hat\zeta(s,y)|}\hat\zeta(s,y)(\bar{\hat\zeta}_t(s,y)-\bar\zeta_t(s,y))\:\d y \d s.
\end{align}
Thus the limit conserves energy since the right hand side of \eqref{eq:energy cons} tends to zero as $\zeta$ tends to the fixed point. Then we know that the solution $(\zeta,\zeta_t)$ is actually in $\D_{T,E}\subsetneq\D_{T,E'}$, and hence, $(\zeta(T),\zeta_t(T))\in X_E$.

We will now show that we can relax the smoothness conditions on the initial values. From the strong formulation of the equation \eqref{eq:strong form} and estimate \eqref{eq:rhs diff} we get, from the Gr\" onwall inequality, the estimates
\begin{align}
	\|\zeta_1(t)-\zeta_2(t)\|_p &\leq \left(\|\zeta_{1,0}-\zeta_{2,0}\|_p+t\|\zeta_{1,t,0}-\zeta_{2,t,0}\|_p\right)e^{\frac12(L_E'+L_E'')t^2},\\
	\|\zeta_{1,t}(t)-\zeta_{2,t}(t)\|_p &\leq \|\zeta_{1,t,0}-\zeta_{2,t,0}\|_p + c\|\zeta_{1,x,0}-\zeta_{2,x,0}\|_p\non &\quad + (L_E'+L_E'')t\left(\|\zeta_{1,0}-\zeta_{2,0}\|_p+t\|\zeta_{1,t,0}-\zeta_{2,t,0}\|_p\right)e^{\frac12(L_E'+L_E'')t^2},\\
	\|\zeta_{1,x}(t)-\zeta_{2,x}(t)\|_p &\leq \frac1c\|\zeta_{1,t,0}-\zeta_{2,t,0}\|_p + \|\zeta_{1,x,0}-\zeta_{2,x,0}\|_p\non &\quad  +\frac{L_E'+L_E''}ct\left(\|\zeta_{1,0}-\zeta_{2,0}\|_p+t\|\zeta_{1,t,0}-\zeta_{2,t,0}\|_p\right)e^{\frac12(L_E'+L_E'')t^2},
\end{align}
for $p=2,\infty$. Thus the strong solution depends continuously on the initial data in the metric induced from $\|(\zeta_0,\zeta_{t0})\| = \|\zeta_0-\zeta^*\|_{H^1(\R)} + \|\zeta_0\|_{W^{1,\infty}(\R)}\| + \|\zeta_{t0}\|_{L^2(\R)} + \|\zeta_{t0}\|_{L^\infty(\R)} + \|W_0(|\zeta_0|)\|_{L^1(\R)}$. Hence we can extend the solution operator to all initial data in $X_E$. The argument can now be repeated to yield a solution in $[0,nT]$ for any $n\in\N$, and hence there exists a global strong solution that conserve energy.
\end{proof}
In the case where the initial data is smooth, $\zeta_0\in C^2(\R)$ and $\zeta_{t0}\in C^1(\R)$, the strong solution can be upgraded to a classical solution $\zeta\in C^2([0,\infty)\times\R)$.
\begin{corollary}
For initial data $(\zeta_0,\zeta_{t0})\in X_E\cap C^2(\R)\times C^1(\R)$ the strong solution $\zeta$ in Theorem \ref{thm:global} belongs to $C^2([0,\infty)\times\R)$ and satisfies the equation \eqref{eq:2NVWlinear} pointwise.
\end{corollary}
\begin{proof}
We now assume that $\zeta$ is a strong solution given implicitly by \eqref{eq:strong form}, and that $\zeta\in C([0,\infty),W^{1,\infty}(\R))\cap C^1([0,\infty),L^\infty(\R))$. Note that we have
\begin{equation}
	W_0(s) = W_0(0) + \frac12W_0^{(2)}(0)s^2 + \frac16W_0^{(3)}(0)s^3 + \int_0^s\frac16W_0^{(4)}(u)(s-u)^3\:\d u,\\
\end{equation}
and thus
\begin{equation}
	\frac{\d}{\d s}\frac{W_0'(s)}{s} = \frac12W_0^{(3)}(0) - \frac1{s^2}\int_0^s\frac12W_0^{(4)}(u)(s-u)^2\:\d u + \frac1s\int_0^sW_0^{(4)}(u)(s-u)\:\d u.
\end{equation}
One can then readily verify that
\begin{align}
	\left|\frac{\d}{\d s}\frac{W_0'(s)}{s}(s)\right| &\leq \frac12\left|W_0^{(3)}(0)\right| + \frac32\left\|W_0^{(4)}\right\|_{L^\infty([0,s])}s,\\
	\left|\frac{\d}{\d s}\frac{W_0'(s)}{s}(s_1) - \frac{\d}{\d s}\frac{W_0'(s)}{s}(s_2)\right| &\leq \frac52 \left\|W_0^{(4)}\right\|_{L^\infty([0,\max\{s_1,s_2\}])}\left|s_1-s_2\right|.
\end{align}
Thus $\zeta\mapsto \frac{W_0'(|\zeta|)}{|\zeta|}\zeta$ is $C^1$ in the open unit disk in two dimensions. Direct computation then yields that for
\begin{equation}
	I(t,x) = \int_0^t\int_{x-c(t-\tau)}^{x+c(t-\tau)}\frac{W_0'(|\zeta(\tau,y)|)}{|\zeta(\tau,y)|}\zeta(\tau,y)\:\d y\d\tau
\end{equation}
the first derivatives are
\begin{align}
	\frac{\partial}{\partial t} I(t,x) &= c\int_0^t \frac{W_0'(|\zeta(\tau,x+c(t-\tau))|)}{|\zeta(\tau,x+c(t-\tau))|}\zeta(\tau,x+c(t-\tau))\non &\qquad + \frac{W_0'(|\zeta(\tau,x-c(t-\tau))|)}{|\zeta(\tau,x-c(t-\tau))|}\zeta(\tau,x-c(t-\tau))\:\d\tau,\\
\frac{\partial}{\partial x} I(t,x) &= \int_0^t 	
	\frac{W_0'(|\zeta(\tau,x+c(t-\tau))|)}{|\zeta(\tau,x+c(t-\tau))|}\zeta(\tau,x+c(t-\tau))\non &\qquad - \frac{W_0'(|\zeta(\tau,x-c(t-\tau))|)}{|\zeta(\tau,x-c(t-\tau))|}\zeta(\tau,x-c(t-\tau))\:\d\tau,
\end{align}
which are continuous since $\zeta$ is continuous by assumption. Hence $\zeta\in C^1([0,\infty)\times\R)$. Similarly the second derivatives are
\begin{align}
	\frac{\partial^2}{\partial t^2} I(t,x) &= 2c\frac{W_0'(|\zeta(t,x)|)}{|\zeta(t,x)|}\zeta(t,x)\non &\quad+ c^2\int_0^t\frac{\partial}{\partial x}\frac{W_0'(|\zeta(\tau,x+c(t-\tau))|)}{|\zeta(\tau,x+c(t-\tau))|}\zeta(\tau,x+c(t-\tau))\non &\qquad- \frac{\partial}{\partial x}\frac{W_0'(|\zeta(\tau,x-c(t-\tau))|)}{|\zeta(\tau,x-c(t-\tau))|}\zeta(\tau,x-c(t-\tau))\:\d\tau,\\
	\frac{\partial^2}{\partial t\partial x}I(t,x) &= c\int_0^t\frac{\partial}{\partial x}\frac{W_0'(|\zeta(\tau,x+c(t-\tau))|)}{|\zeta(\tau,x+c(t-\tau))|}\zeta(\tau,x+c(t-\tau))\non &\qquad + \frac{\partial}{\partial x}\frac{W_0'(|\zeta(\tau,x-c(t-\tau))|)}{|\zeta(\tau,x-c(t-\tau))|}\zeta(\tau,x-c(t-\tau))\:\d\tau,\\
	\frac{\partial^2}{\partial x^2}I(t,x) &=\int_0^t\frac{\partial}{\partial x}\frac{W_0'(|\zeta(\tau,x+c(t-\tau))|)}{|\zeta(\tau,x+c(t-\tau))|}\zeta(\tau,x+c(t-\tau))\non &\qquad - \frac{\partial}{\partial x}\frac{W_0'(|\zeta(\tau,x-c(t-\tau))|)}{|\zeta(\tau,x-c(t-\tau))|}\zeta(\tau,x-c(t-\tau))\:\d\tau,
\end{align}
which are continuous. Thus $\zeta$ given by 
\begin{equation}
	\zeta(t,x) = \frac12\left(\zeta(t,x+ct)+\zeta(t,x-ct)\right) + \frac1{2c}\int_{x-ct}^{x+ct}\zeta_{t0}(y)\:\d y -\frac1{2c}I(t,x),
\end{equation}
is a sum of $C^2([0,\infty)\times\R)$ functions and is thus $C^2([0,\infty)\times\R)$. We get
\begin{align}
	\frac{\partial^2}{\partial t^2}\zeta(t,x)-c^2\frac{\partial^2}{\partial t^2}\zeta(t,x) &= -\frac{\partial^2}{\partial t^2}\frac1{2c}I(t,x) + c^2\frac{\partial^2}{\partial x^2}\frac1{2c}I(t,x) \non
	&= - \frac{W_0'(|\zeta(t,x)|)}{|\zeta(t,x)|}\zeta(t,x)
\end{align}
pointwise.
\end{proof}
We will only give a local existence theorem for \eqref{eq:2NVW}. The main difficulty for \eqref{eq:2NVW} compared to \eqref{eq:2NVWlinear} is that \eqref{eq:2NVW} is quasilinear, that is the terms $c(\psi)^2\psi_{xx}$ and $c(\psi)^2s_{xx}$ are nonlinear. The method used in the proof of Theorem \ref{thm:global} relied heavily on \eqref{eq:2NVWlinear} being semilinear.
\begin{theorem}
\label{thm:local}
There exists a unique short time solution of \eqref{eq:2NVW} for smooth initial data $(\psi_0,\psi_{t0},s_0,s_{t0})$, where $s_0$ is bounded away from $0$ and the energy is finite.
\end{theorem}
\begin{proof}
Local existence, uniqueness, and continuous dependence on initial data follows from the standard approach taken to semigroups of nonlinear evolution equations. Here we will diverge from the standard semigroup approach by considering linear operators on Banach spaces, and instead solve the linear equations by characteristics.

In particular, inspired by \cite{K2}, the system \eqref{eq:2NVW} can be rewritten as a quasilinear symmetric hyperbolic system and the results in \cite{K} applied. Indeed, introduce the variables
\begin{subequations}
\begin{align}
	\phi &= \psi_t,\\
	v &= s_t,\\
	\omega &= c(\psi)\psi_x,\\
	r &= c(\psi)s_x,
\end{align}
\end{subequations}
the system \eqref{eq:2NVW} can be written as a quasilinear symmetric hyperbolic system
\begin{subequations}
\label{eq:symhyp}
\begin{align}
	\psi_t &= \phi,\\
	s_t &= v,\\
	\phi_t-c(\psi)\omega_x &= -\frac2s\left(\phi v-\omega r\right)+\frac{c'(\psi)}{c(\psi)}r^2,\\
	v_t-c(\psi)r_x &=s\left(\phi^2-\omega^2\right)+\frac{c'(\psi)}{c(\psi)}\omega r-W_0'(s),\\
	\omega_t-c(\psi)\phi_x &= \frac{c'(\psi)}{c(\psi)}\phi\omega,\\
	r_t-c(\psi)v_x &= \frac{c'(\psi)}{c(\psi)}\phi r.
\end{align}
\end{subequations}
We introduce the Banach space $X=W^{1,\infty}(\R,\R^6)$, and let 
\begin{align}
	X_{E,L} = \bigg\{&U=(\psi,s,\phi,v,\omega,r)\in X \mid \|U\|_{W^{1,\infty}}\leq L, \left\|\frac1s\right\|_{L^\infty(\R)}\leq L,\non &\int_{\R}\frac12s^2\left(\phi^2+\omega^2\right)+\frac12\left(v^2+r^2\right)+W_0(s)\:\d x \leq E \bigg\}.
\end{align}
Note that for elements of $X_{E_L}$ the variable $s$ is of definite sign since $s$ is continuous and  $\frac1s$ is bounded. We define our solution space $\D_{T,E,L}$ as follows
\begin{equation}
	\D_{T,E,L} = C([0,T],X_{E,L})\cap C^1([0,T],L^\infty(\R,\R^6)).
\end{equation}
We equip $\D_{T,E,L}$ with the metric
\begin{equation}
	d_{\D_{T,E,L}}(U_1,U_2) = \sup_{0\leq t\leq T}\|U_1(t)-U_2(t)\|_{W^{1,\infty}(\R,\R^6)}+ \sup_{0\leq t\leq T}\|U_{1,t}(t)-U_{2,t}(t)\|_{L^\infty(\R,\R^6)},
\end{equation}
which renders $\D_{T,E,L}$ a complete metric space. Formally, we can formulate \eqref{eq:symhyp} as
\begin{equation}
	U_t-c(\psi)AU_x = F(U),
\end{equation}
where $A$ is a constant symmetric hyperbolic matrix with eigenvalues $-1,0,1$. We want to show existence of short time solutions from a fix point argument. To be able to make a contraction we will further restrict our space for approximate solutions to
\begin{equation}
	\D_{T,E,L}^{\mathrm{lin}} = \left\{U\in\D_{T,E,L}\mid U(t)\in W^{2,\infty}(\R,\R^6), \|U_{xx}(t)\|_{L^\infty(\R,\R^6)}\leq L\right\}.
\end{equation}

Let $0<E<E', 0<L<L'$ and $U_0\in X_{E,L}$ with $\|U_{0xx}\|_\infty\leq L$ be the initial data, then for any $\hat U\in\D_{T,E',L'}^{\mathrm{lin}}$ with $\hat U(0) = U_0$, the linear system of transport equations
\begin{equation}
U_t -c(\hat\psi)AU_x = F(\hat U)
\end{equation}
can be solved by characteristics. Define the backward characteristics $x_\pm$ at time $\tau$ from point $(t,x)$ by
\begin{equation}
	\frac{\d}{\d\tau}x_\pm(\tau;t,x) = \mp c\bigg(\hat\psi(\tau,x_\pm(\tau;t,x))\bigg), 
\end{equation}
and note that $x_\pm(\tau;t,x) = x_\pm(\tau;s,x_\pm(s;t,x))$ and $x_\pm(t;t,x))=x$. Then, by taking the derivative with respect to $s$, we get the following
\begin{equation}
	\frac{\partial}{\partial t}x_\pm(\tau;t,x) \mp c(\hat\psi(\tau,x_\pm(\tau;t,x)))\frac{\partial}{\partial x}x_\pm(\tau;t,x) = 0.
\end{equation}
Thus
\begin{align}
	|x_\pm(\tau_1;t,x)-x_\pm(\tau_2;t,x)| &\leq \|c\|_\infty|\tau_2-\tau_1|,\\
	|x_\pm(\tau;t,x_1)-x_\pm(\tau;t,x_2)| &\leq \mathrm{exp}\left\{\left\|\frac1c\right\|_\infty L'|t-\tau|\right\}|x_1-x_2|,\\
	|x_+(\tau;t,x)-x_-(\tau;t,x)| &\leq 2\|c\|_\infty|t-\tau|,\\
	\left|\frac{\partial}{\partial t}x_\pm(\tau;t,x)\right| &\leq \|c\|_\infty\:\mathrm{exp}\left\{\left\|\frac1c\right\|_\infty L'|t-\tau|\right\}.
\end{align}
We need to show that the Duhamel operator $DF(\hat U)(t) = \int_0^t T_\tau(t-\tau)F(\hat U(\tau))\:\d\tau$, where $T_{\tau}$ is the solution operator of $U_t(t)-c(\psi(t+\tau))AU_x(t) = 0$, satisfies $DF(\hat U)\in C([0,\infty),W^{2,\infty}(\R,\R^6))$, and $\|DF(\hat U_1)(t)-DF(\hat U_2)(t)\|_{W^{1,\infty}(\R)}\leq C(t,E',L')\|\hat U_1-\hat U_2\|_{W^{1,\infty}(\R)}$ for $\hat U_1,\hat U_2\in \D_{T,E',L'}^{\mathrm{lin}}$. Since
\begin{equation}
	\frac{\partial}{\partial x}\frac{c'(\psi)}{c(\psi)} = \frac{c''(\psi)c(\psi)-c'(\psi)^2}{c(\psi)^2}\psi_x,
\end{equation}
and  $F$ is a smooth function, we have that for $\hat U\in \D_{T,E',L'}^{\mathrm{lin}}$ there is
\begin{equation}
	\|F(\hat U(t))\|_{W^{2,\infty}(\R)} \leq C_{E',L'}.
\end{equation}
Hence the Duhamel operator satisfies $\|DF(\hat U)(t)\|_{W^{2,\infty}(\R)}\leq C_{E',L'}t$. Note that
\begin{equation}
\frac\partial{\partial\tau}\left(x_\pm^1(\tau;t,x_1)-x_\pm^2(\tau;t,x_2)\right)=\mp\left(c(\hat\psi_1(\tau,x_\pm^1(\tau;t,x_1)))-c(\hat\psi_2(\tau,x_\pm^2(\tau;t,x_2)))\right),
\end{equation}
and thus
\begin{align}
	|x_\pm^1(\tau;t,x_1)-x_\pm^2(\tau;t,x_2)| &\leq |x_1-x_2|e^{\|c'\|_\infty L'(t-\tau)}\non &\quad+\frac1{L'} \left(e^{\|c'\|_\infty L'(t-\tau)}-1\right)\|\hat\psi_1-\hat\psi_2\|_{L^\infty([0,t],L^\infty(\R))},\\
	\left|\frac{\partial}{\partial x}x_\pm^1(\tau;t,x)-\frac{\partial}{\partial x}x_\pm^2(\tau;t,x)\right|  &\leq C_{L'}\mathrm{exp}\left\{k_{L'}(t-\tau)\right\}\sup_{0\leq\tau\leq t}\|\hat\psi_1-\hat\psi_2\|_{W^{1,\infty}(\R)},
\end{align}
where $C_{L'}$ and $k_{L'}$ depends on $c$ and $L'$ only. Thus, since $F$ is a smooth function and $\|\hat U_1\|_{W^{2,\infty}(\R)}, \|\hat U_2\|_{W^{2,\infty}(\R)} \leq L'$, 
\begin{align}
	&\left|\int_0^tF_j(\hat U_1)(\tau,x^1_\pm(\tau;t,x))-F_j(\hat U_2)(\tau,x^2_\pm(\tau;t,x))\:\d\tau\right|\non
	&\leq \left|\int_0^tF_j(\hat U_1)(\tau,x^1_\pm(\tau;t,x))-F_j(\hat U_2)(\tau,x^1_\pm(\tau;t,x))\:\d\tau\right| \non &\quad+ \left|\int_0^tF_j(\hat U_2)(\tau,x^1_\pm(\tau;t,x))-F_j(\hat U_2)(\tau,x^2_\pm(\tau;t,x))\:\d\tau\right|\non
	&\leq \left|\int_0^tF_j(\hat U_1)(\tau,x^1_\pm(\tau;t,x))-F_j(\hat U_2)(\tau,x^1_\pm(\tau;t,x))\:\d\tau\right| \non &\quad+ t \bar C(L',E')\|\hat\psi_1-\hat\psi_2\|_{L^\infty([0,t],L^\infty(\R)}\non
	&\leq t\tilde C(L',E')\|\hat U_1 -\hat U_2\|_{L^\infty([0,t],L^\infty(\R,\R^6)}
\end{align}
To compute $\frac\partial{\partial x}\left(DF(\hat U_1)(t)-DF(\hat U_2)(t)\right)$ we estimate
\begin{align}
	&\quad\left|\frac\partial{\partial x}\int_0^t F_j(\hat U_1)(\tau,x_\pm^1(\tau;t,x)) - F_j(\hat U_2)(\tau,x_\pm^2(\tau;t,x)) \:\d\tau\right|\non &\leq \left|\int_0^t\left(\nabla_UF_j(\hat U_1)\frac{\partial \hat U_1}{\partial x}-\nabla_UF_j(\hat U_2)\frac{\partial \hat U_2}{\partial x}\right)(\tau,x_\pm^1(\tau;t,x)) \frac{\partial}{\partial x}x_\pm^1(\tau;t,x)\:\d\tau\right| \non
	&\quad + \left|\int_0^t\nabla_UF_j(\hat U_2)\frac{\partial \hat U_2}{\partial x}(\tau,x_\pm^1(\tau;t,x))\frac{\partial}{\partial x}\left(x_\pm^1(\tau;t,x)-x_\pm^2(\tau;t,x)\right)\:\d\tau \right|\non
	&\quad +\bigg|\int_0^t\left(\nabla_UF_j(\hat U_2)\frac{\partial \hat U_2}{\partial x}(\tau,x_\pm^1(\tau;t,x))-\nabla_UF_j(\hat U_2)\frac{\partial \hat U_2}{\partial x}(\tau,x_\pm^2(\tau;t,x))\right)\non &\qquad\qquad\times\frac\partial{\partial x}x_\pm^2(\tau;t,x)\:\d\tau\bigg|\non
&\leq C_{E',L'}\left(\mathrm{exp}\left\{\left\|\frac1c\right\|_\infty L't\right\}-1\right)\sup_{0\leq\tau\leq t}\|\hat U_1(\tau)-\hat U_2(\tau)\|_{W^{1,\infty}(\R)}\non
	&\quad + C_{E',L'}\:(\mathrm{exp}\left\{K_{L'}t\right\}-1)\sup_{0\leq\tau\leq t}\|\hat \psi_1(\tau)-\hat \psi_2(\tau)\|_{W^{1,\infty}(\R)}\non
	&\quad + C_{E',L'}\:(\mathrm{exp}\left\{K_{L'}t\right\}-1)\sup_{0\leq\tau\leq t}\|\hat \psi_1(\tau)-\hat \psi_2(\tau)\|_{L^\infty(\R)}.
\end{align}
That is, we have that
\begin{equation}
	\|U_1(t)-U_2(t)\|_{W^{1,\infty}(\R)} \leq C_{L',E'}(\mathrm{exp}\{k_{L'}t\}-1)\|\hat U_1-\hat U_2\|_{W^{1,\infty}(\R)},
\end{equation}
and hence for $T$ small enough we have that 
\begin{equation}
	\sup_{0\leq t\leq T}\|U_1(t)-U_2(t)\|_{W^{1,\infty}(\R)} < \sup_{0\leq t\leq T}\|\hat U_1(t)-\hat U_2(t)\|_{W^{1,\infty}(\R)}.
\end{equation}
Since $U$ is continuously differentiable, we can do similar estimates for $\sup_{0\leq t\leq T}\|U_{1,t}(t)-U_{2,t}(t)\|_{L^\infty(\R)}$ directly from the linear system. We need  that the energy of the solution is bounded by $E'$. Recall that $E(t)=\int_\R\mc E(t)\:\d x = \int_\R s(t)^2(\phi(t)^2+\omega(t)^2)+v(t)^2 + r(t)^2 + W_0(s(t)) \:\d x$, and thus
\begin{align}
	\frac\partial{\partial t}\mc E - \frac\partial{\partial x}c(\hat\psi)\left(s^2\phi\omega+vr\right) &= s\phi^2\hat v+v\hat s\hat\phi^2-\frac2{\hat s}s^2\phi\hat\phi\hat v\non
	&\quad + s\omega^2\hat v - v\hat s\hat\omega^2+ W_0'(s)\hat v-W_0'(\hat s)v\non
	&\quad +\frac{c'(\hat\psi)}{c(\hat\psi)}\left(s^2\omega\hat\omega(\hat\phi-\phi) +\hat r(\hat\phi r-\phi\hat r) + v(\hat\omega\hat r-\omega r)\right)\non
	&\quad +\frac{c'(\hat\psi)}{c(\hat\psi)}2s\phi\left(\frac s{\hat s}\hat\omega\hat r-\omega c(\hat\psi)s_x\right).
\end{align}
Integration yields
\begin{align}
	\frac\d{\d t}E(t) &\leq 2{L'}^2E(t)+ 2{L'}^2E' +2{L'}^2\sqrt{E'}\sqrt{E(t)} +\|W_0'(s)\hat v\|_1 + \|W_0'(\hat s)v\|_1\non
	&\quad + K\left({L'}^3\sqrt{E(t)}\sqrt{E'} +  4{L'}\sqrt{E(t)}\sqrt{E'}\right)\non
	&\quad + K\left(2{L'}^3\sqrt{E(t)}\sqrt{E'}+2E(t)\|c\|_\infty L'\right)
\end{align}
By Definition \ref{def:W_0} we have that
\begin{align}
	\|W_0'(s)\hat v\|_1 &= \left\|W_0'(s)\right\|_2\left\|\hat v\right\|_2 \non
	&\leq \sqrt{k_{E(t)}\|W_0(s)\|_1}\sqrt{E'}\non
	&\leq \sqrt{k_{E'}E(t)}\sqrt{E'}, 
\end{align}
and hence for $t$ small enough, depending on $E'$, $L'$, $W_0$, and $c$ only, $E(t)\leq E'$.

Hence the sequence $U_n$, where $U_{n+1}$ is the solution to the linear equation with $U_n$ inserted, is a Cauchy sequence in $\D_{T,E',L'}$, and hence there is a solution of the equations. Moreover, the limit will satisfy $\mc E_t-(c(\psi)^2\mc F)_x = 0$, and thus conserves energy.
\end{proof}

\end{document}